\documentclass[11pt]{amsart}
\usepackage{amssymb}
\usepackage{mathrsfs}
\usepackage{graphicx}
\usepackage[all]{xypic}
\newcommand{\mo}{\mathrm{mod}\,}
\newcommand{\ind}{\mathrm{ind}\,}
\newcommand{\Hom}{\mathrm{Hom}}
\newcommand{\Tr}{\mathrm{Tr}}
\newcommand{\op}{\mathrm{op}}
\newcommand{\soc}{\mathrm{soc}}
\newcommand{\End}{\mathrm{End}}
\newcommand{\Ext}{\mathrm{Ext}}
\newcommand{\rad}{\mathrm{rad}}
\newcommand{\add}{\mathrm{add}}
\newcommand{\Tor}{\mathrm{Tor}}
\newcommand{\ann}{\mathrm{ann}}

\newcommand{\pd}{\mathrm{pd}}
\newcommand{\id}{\mathrm{id}}
 \newtheorem{thm}{Theorem}[section]

 \newtheorem{prop}[thm]{Proposition}
 \theoremstyle{definition}
 
 \theoremstyle{remark}

 \numberwithin{equation}{section}

\begin{document}

\title[The number of terms in the middle of almost split sequences]
 {On the number of terms in the middle \\[2mm] of almost split sequences over \\[2mm] cycle-finite artin algebras}

\author[Malicki]{Piotr Malicki}
\address{Faculty of Mathematics and Computer Science, Nicolaus Copernicus University, Chopina 12/18, 87-100 Toru\'n, Poland}
\email{pmalicki@mat.uni.torun.pl}

\thanks{This work was completed with the support of the research grant No. 2011/02/A/ST1/00216 of the Polish National Science Center
and the CIMAT Guanajuato, M\'exico.}
\author[de la Pe\~na]{Jos\'e Antonio de la Pe\~na}
\address{Centro de Investigaci\'on en Mathem\'aticas (CIMAT), Guanajuato, M\'exico}
\email{jap@cimat.mx}
%
\author[Skowro\'nski]{Andrzej Skowro\'nski}
\address{Faculty of Mathematics and Computer Science, Nicolaus Copernicus University, Chopina 12/18, 87-100 Toru\'n, Poland}
\email{skowron@mat.uni.torun.pl}



\begin{abstract}
We prove that the number of terms in the middle of an almost split sequence in the module category of a cycle-finite
artin algebra is bounded by $5$.
\end{abstract}

\maketitle
\section{Introduction and the main result}\label{intro}
Throughout this paper, by an algebra is meant an artin algebra over a fixed commutative artin ring $K$, which we moreover assume (without loss of generality)
to be basic and indecomposable. For an algebra $A$, we denote by $\mo A$ the category of finitely generated right $A$-modules, by $\ind A$ the full
subcategory of $\mo A$ formed by the indecomposable modules, by $\Gamma_A$ the Auslander-Reiten quiver of $A$, and by $\tau_A$ and $\tau^{-1}_A$
the Auslander-Reiten translations $D\Tr$ and $\Tr D$, respectively. We do not distinguish between a module in $\ind A$ and the vertex of $\Gamma_A$
corresponding to it. The Jacobson radical $\rad_A$ of $\mo A$ is the ideal generated by all nonisomorphisms between modules in $\ind A$, and the
infinite radical $\rad^{\infty}_A$ of $\mo A$ is the intersection of all powers $\rad^i_A$, $i\geq 1$, of $\rad_A$.
By a theorem of M. Auslander \cite{Au}, $\rad_A^{\infty}=0$ if and only if $A$ is of finite representation type, that is, $\ind A$ admits only a finite
number of pairwise nonisomorphic modules. On the other hand, if $A$ is of infinite representation type then $(\rad_A^{\infty})^2\neq 0$, by a theorem
proved in \cite{CMMS}.

A prominent role in the representation theory of algebras is played by almost split sequences introduced by M. Auslander and I. Reiten in \cite{AR1}
(see \cite{ARS} for general theory and applications). For an algebra $A$ and a nonprojective module $X$ in $\ind A$, there is an almost split sequence
\[ 0\to \tau_AX\to Y\to X\to 0,\]
with $\tau_AX$ a noninjective module in $\ind A$ called the Auslander-Reiten translation of $X$. Then we may associate to $X$ the numerical invariant
$\alpha(X)$ being the number of summands in a decomposition $Y=Y_1\oplus\ldots\oplus Y_r$ of $Y$ into a direct sum of modules in $\ind A$. Then
$\alpha(X)$ measures the complication of homomorphisms in $\mo A$ with domain $\tau_AX$ and codomain $X$. Therefore, it is interesting to study the
relation between an algebra $A$ and the values $\alpha(X)$ for all modules $X$ in $\ind A$ (we refer to \cite{AR2}, \cite{BaBr}, \cite{BR},
\cite{Kr}, \cite{Li3}, \cite{PeSk}, \cite{PTa}, \cite{PoSk}, \cite{SW}, \cite{WW} for some results in this direction).
In particular, it has been proved by R. Bautista and S. Brenner in \cite{BaBr} that, if $A$ is of finite representation type and $X$ a nonprojective
module in $\ind A$, then $\alpha(X)\leq 4$, and if $\alpha(X)=4$ then the middle term $Y$ of an almost split sequence in $\mo A$ with the right term $X$
admits an indecomposable projective-injective direct summand $P$, and hence $X=P/\soc(P)$. In \cite{Li3} S. Liu generalized this result by showing that
the same holds for any nonprojective module $X$ in $\ind A$ over an algebra $A$ provided $\tau_AX$ has a projective predecessor and $X$ has an injective
successor in $\Gamma_A$, as well as for $X$ lying on an oriented cycle in $\Gamma_A$ (see also \cite{Kr}).
It has been conjectured by S. Brenner that $\alpha(X)\leq 5$ for any nonprojective module $X$ in $\ind A$ for an arbitrary tame finite dimensional
algebra $A$ over an algebraically closed field $K$. In fact, it is expected that this also holds for nonprojective indecomposable modules over
arbitrary generically tame (in the sense of \cite{CB1}, \cite{CB2}) artin algebras.

The main aim of this paper is to prove the following theorem which gives the affirmative answer for the above conjecture in the case of cycle-finite
artin algebras. \\

\noindent \textbf{Theorem.} \textit{Let $A$ be a cycle-finite algebra and $X$ be a nonprojective module in $\ind A$, and
\[ 0\to \tau_AX\to Y\to X\to 0\]
be the associated almost split sequence in $\mo A$. The following statements hold.
\begin{enumerate}
\renewcommand{\labelenumi}{\rm(\roman{enumi})}
\item $\alpha(X)\leq 5$.
\item If $\alpha(X)=5$ then $Y$ admits an indecomposable projective-injective direct summand $P$, and hence $X=P/\soc(P)$.
\end{enumerate}}
\smallskip

We would like to mention that, for finite dimensional cycle-finite algebras $A$ over an algebraically closed field $K$, the theorem was proved by
J. A. de la Pe\~na and M. Takane \cite[Theorem 3]{PTa}, by application of spectral properties of Coxeter transformations of algebras and results
established in \cite{Li3}.

Let $A$ be an algebra. Recall that a cycle in $\ind A$ is a sequence
\[ X_0 \buildrel {f_1}\over {\hbox to 6mm{\rightarrowfill}} X_1 \to \cdots \to X_{r-1} \buildrel {f_r}\over {\hbox to 6mm{\rightarrowfill}} X_r=X_0 \]
of nonzero nonisomorphisms in $\ind A$ \cite{Ri1}, and such a cycle is said to be finite if the homomorphisms $f_1,\ldots, f_r$ do not belong to $\rad_A^{\infty}$.
Then, following \cite{AS2}, \cite{Sk5}, an algebra $A$ is said to be cycle-finite if all cycles in $\ind A$ are finite. The class of cycle-finite algebras contains
the following distinguished classes of algebras: the algebras of finite representation type, the hereditary algebras of Euclidean type \cite{DR1}, \cite{DR2},
the tame tilted algebras \cite{HRi1}, \cite{Ke1}, \cite{Ri1}, the tame double tilted algebras \cite{RS1}, the tame generalized double tilted algebras \cite{RS2},
the tubular algebras \cite{Ri1}, the iterated tubular algebras \cite{PTo}, the tame quasi-tilted algebras \cite{LS}, \cite{Sk8},
the tame generalized multicoil algebras \cite{MS}, the algebras with cycle-finite derived categories \cite{AS1}, and the strongly simply connected
algebras of polynomial growth \cite{Sk6}.
On the other hand, frequently an algebra $A$ admits a Galois covering $R\to R/G = A$, where $R$ is a cycle-finite locally bounded category and $G$
is an admissible group of automorphisms of $R$, which allows to reduce the representation theory of $A$ to the representation theory of cycle-finite
algebras being finite convex subcategories of $R$ (see \cite{DS}, \cite{PeSk}, \cite{Sk7} for some general results).
For example, every finite dimensional selfinjective algebra of polynomial growth over an algebraically closed field admits a canonical standard form
$\overline A$ (geometric socle deformation of $A$) such that $\overline A$ has a Galois covering $R\to R/G = \overline A$, where $R$ is a cycle-finite
selfinjective locally bounded category and $G$ is an admissible infinite cyclic group of automorphisms of $R$, the Auslander-Reiten quiver
$\Gamma_{\overline A}$ of $\overline A$ is the orbit quiver $\Gamma_R/G$ of $\Gamma_R$, and the stable Auslander-Reiten quivers of $A$ and $\overline A$
are isomorphic (see \cite{Sk1} and \cite{Sk9}). Recall also that, a module $X$ in $\ind A$ which does not lie on a cycle in $\ind A$ is called directing,
and its support algebra is a tilted algebra, by a result of C. M. Ringel \cite{Ri1}. Moreover, it has been proved independently by L. G. Peng - J. Xiao
\cite{PX} and A. Skowro\'nski \cite{Sk3} that the Auslander-Reiten quiver $\Gamma_A$ of an algebra $A$ admits at most finitely many $\tau_A$-orbits
containing directing modules.

\section{Preliminary results}
Let $H$ be an indecomposable hereditary algebra and $Q_H$ the valued quiver of $H$. Recall that the vertices of $Q_H$ are the
numbers $1, 2, \ldots, n$ corresponding to a complete set $S_1, S_2, \ldots, S_n$ of pairwise nonisomorphic simple modules in
$\mo H$ and there is an arrow from $i$ to $j$ in $Q_H$ if $\Ext^1_H(S_i,S_j)\neq 0$, and then to this arrow is assigned the
valuation $(\dim_{\End_H(S_j)}\Ext^1_H(S_i,S_j),$ $\dim_{\End_H(S_i)}\Ext^1_H(S_i,S_j))$.
Recall also that the Auslander-Reiten quiver $\Gamma_H$ of $H$ has a disjoint union decomposition of the form
\[\Gamma_H = \mathcal{P}(H) \vee \mathcal{R}(H) \vee \mathcal{Q}(H),\]
where $\mathcal{P}(H)$ is the preprojective component containing all indecomposable projective $H$-modules, $\mathcal{Q}(H)$ is the
preinjective component containing all indecomposable injective $H$-modules, and $\mathcal{R}(H)$ is the family of all regular
components of $\Gamma_H$. More precisely, we have:
\begin{itemize}
\item[$\bullet$] if $Q_H$ is a Dynkin quiver, then $\mathcal{R}(H)$ is empty and
$\mathcal{P}(H)=\mathcal{Q}(H)$;
\item[$\bullet$] if $Q_H$ is a Euclidean quiver, then $\mathcal{P}(H)\cong (-\mathbb{N})Q^{\op}_H$, $\mathcal{Q}(H)\cong \mathbb{N}Q^{\op}_H$ and
$\mathcal{R}(H)$ is a strongly separating infinite family of
stable tubes;
\item[$\bullet$] if $Q_H$ is a wild quiver, then $\mathcal{P}(H) \cong (-\mathbb{N})Q^{\op}_H$, $\mathcal{Q}(H)\cong \mathbb{N}Q^{\op}_H$ and
$\mathcal{R}(H)$ is an infinite family of components of type
$\mathbb{ZA}_{\infty}$.
\end{itemize}
Let $T$ be a tilting module in $\mo H$ and $B=\End_H(T)$ the associated tilted algebra. Then the tilting $H$-module $T$
determines the torsion pair $(\mathcal{F}(T), \mathcal{T}(T))$ in $\mo H$, with the torsion-free part $\mathcal{F}(T)=\{X \in \mo H |$ $\Hom_H(T,X)=0\}$
and the torsion part $\mathcal{T}(T)=\{X \in \mo H | \Ext^1_H(T,X)=0\}$, and the splitting torsion pair $(\mathcal{Y}(T), \mathcal{X}(T))$ in $\mo B$,
with the torsion-free part $\mathcal{Y}(T)=\{Y \in\mo B|\Tor^B_1(Y,T)=0\}$ and the torsion part $\mathcal{X}(T)=\{Y \in \mo B| Y \otimes_B T=0\}$.
Then, by the Brenner-Butler theorem, the functor $\Hom_H(T,-): \mo H \to \mo B$ induces an equivalence of $\mathcal{T}(T)$ with $\mathcal{Y}(T)$, and the
functor $\Ext^1_H(T,-): \mo H \to \mo B$ induces an equivalence of $\mathcal{F}(T)$ with $\mathcal{X}(T)$ (see \cite{BrBu}, \cite{HRi1}). Further, the images
$\Hom_H(T,I)$ of the indecomposable injective modules $I$ in $\mo H$ via the functor $\Hom_H(T,-)$ belong to one component $\mathcal{C}_T$ of $\Gamma_B$,
called the connecting component of $\Gamma_B$ determined by $T$, and form a faithful section $\Delta_T$ of $\mathcal{C}_T$, with $\Delta_T$ the opposite valued
quiver $Q^{\op}_H$ of $Q_H$. Recall that a full connected valued subquiver $\Sigma$ of a component $\mathcal{C}$ of $\Gamma_B$ is called a section
(see \cite[(VIII.1)]{ASS}) if $\Sigma$ has no oriented cycles, is convex in $\mathcal{C}$, and intersects each $\tau_B$-orbit of $\mathcal{C}$ exactly once.
Moreover, the section $\Sigma$ is faithful provided the direct sum of all modules lying on $\Sigma$ is a faithful $B$-module. The section $\Delta_T$ of the
connecting component $\mathcal{C}_T$ of $\Gamma_B$ has the distinguished property: it connects the torsion-free part $\mathcal{Y}(T)$ with the torsion part
$\mathcal{X}(T)$, because every predecessor in $\ind B$ of a module $\Hom_H(T,I)$ from $\Delta_T$ lies in $\mathcal{Y}(T)$ and every successor of
$\tau^-_B \Hom_H(T,I)$ in $\ind B$ lies in $\mathcal{X}(T)$. We note that, by a result proved in \cite{Li2} and \cite{Sk2}, an algebra
$A$ is a tilted algebra if and only if $\Gamma_A$ admits a component $\mathcal C$ with a faithful section $\Delta$ such that
$\Hom_A(X,\tau_AY)=0$ for all modules $X$ and $Y$ from $\Delta$. We refer also to \cite{JMS} for another characterization of tilted
algebras involving short chains of modules.

The following proposition is a well-known fact.
\begin{prop}\label{prop21}
Let $H$ be a hereditary algebra of Euclidean type. Then, for any nonprojective indecomposable module $X$ in $\mo H$, we have $\alpha(X)\leq 4$.
\end{prop}
An essential role in the proof of the main theorem will be played by the following theorem.
\begin{thm}\label{th22}
Let $A$ be a cycle-finite algebra, $\mathcal C$ a component of $\Gamma_A$, and $\mathcal D$ be an acyclic left stable full translation subquiver
of $\mathcal C$ which is closed under predecessors. Then there exists a hereditary algebra $H$ of Euclidean type and a tilting module $T$ in $\mo H$
without nonzero preinjective direct summands such that for the associated tilted algebra $B=\End_H(T)$ the following statements hold.
\begin{enumerate}
\renewcommand{\labelenumi}{\rm(\roman{enumi})}
\item $B$ is a quotient algebra of $A$.
\item The torsion-free part $\mathcal{Y}(T)\cap\mathcal{C}_T$ of the connecting component $\mathcal{C}_T$ of $\Gamma_B$ determined by $T$ is a full
translation subquiver of $\mathcal D$ which is closed under predecessors in $\mathcal C$.
\item For any indecomposable module $N$ in $\mathcal D$, we have $\alpha(N)\leq 4$.
\end{enumerate}
\end{thm}
\begin{proof}
Since $A$ is a cycle-finite algebra, every acyclic module $X$ in $\Gamma_A$ is a directing module in $\ind A$. Hence $\mathcal D$ consists entirely
of directing modules. Moreover, it follows from \cite[Theorem 2.7]{PX} and \cite[Corollary 2]{Sk3}, that $\mathcal D$ has only finitely many $\tau_A$-orbits.
Then, applying \cite[Theorem 3.4]{Li1}, we conclude that there is a finite acyclic valued quiver $\Delta$ such that $\mathcal D$ contains a full translation
subquiver $\Gamma$ which is closed under predecessors in $\mathcal C$ and is isomorphic to the translation quiver $\mathbb{N}\Delta$. Therefore, we may
choose in $\Gamma$ a finite acyclic convex subquiver $\Delta$ such that $\Gamma$ consists of the modules $\tau_A^mX$ with $m\geq 0$ and $X$ indecomposable
modules lying on $\Delta$. Let $M$ be the direct sum of all indecomposable modules in $\mathcal C$ lying on the chosen quiver $\Delta$. Let $I$ be the
annihilator $\ann_A(M)=\{a\in A\mid Ma=0\}$ of $M$ in $A$, and $B=A/I$ the associated quotient algebra. Then $I=\ann_A(\Gamma)$ (see \cite[Lemma 3]{Sk2})
and consequently $\Gamma$ consists of indecomposable $B$-modules. Clearly, $B$ is a cycle-finite algebra, as a quotient algebra of $A$. Now, using the fact
that $\Gamma\subseteq\mathbb{N}\Delta$ and consists of directing $B$-modules, we conclude that $\rad^{\infty}_B(M,M)=0$ and $\Hom_B(M,\tau_BM)=0$. Then,
applying \cite[Lemma 3.4]{Sk4}, we conclude that $H=\End_B(M)$ is a hereditary algebra and the quiver $Q_H$ of $H$ is the dual valued quiver $\Delta^{\op}$
of $\Delta$. Further, since $M$ is a faithful $B$-module with $\Hom_B(M,\tau_BM)=0$, we conclude that $\pd_BM\leq 1$ and
$\Ext^1_B(M,M)\cong D\overline{\Hom}_B(M,\tau_BM)=0$ (see \cite[Lemma VIII.5.1 and Theorem IV.2.13]{ASS}). Moreover, it follows from definition of $M$ that,
for any module $Z$ in $\ind B$ with $\Hom_B(M,Z)\neq 0$ and not on $\Delta$, we have $\Hom_B(\tau^{-1}_BM,Z)\neq 0$. Since $M$ is a faithful module in $\mo B$ there
is a monomorphism $B\to M^s$ for some positive integer $s$. Then $\rad^{\infty}_B(M,M)=0$ implies $\Hom_B(\tau^{-1}_BM,B)=0$, and consequently $\id_BM\leq 1$.
Applying now \cite[Lemma 1.6]{RSS} we conclude that $M$ is a tilting $B$-module. Further, applying the Brenner-Butler theorem (see \cite[Theorem VI.3.8]{ASS}), we
conclude that $M$ is a tilting module in $\mo H^{\op}$ and $B\cong\End_{H^{\op}}(M)$. Since $H$ is a hereditary algebra, $T=D(M)$ is a tilting module in
$\mo H$ with $B\cong\End_{H}(T)$, and consequently $B$ is a tilted algebra of type $Q_H=\Delta^{\op}$. Moreover, the translation quiver $\Gamma$ is the
torsion-free part $\mathcal{Y}(T)\cap\mathcal{C}_T$ of the connecting component $\mathcal{C}_T$ of $\Gamma_B$ determined by the tilting $H$-module $T$
(see \cite[Theorem VIII.5.6]{ASS}). Observe that then $\mathcal{Y}(T)\cap\mathcal{C}_T$ is the image $\Hom_H(T,Q(H))$ of the preinjective component $Q(H)$
of $\Gamma_H$ via the functor $\Hom_H(T,-): \mo H \to \mo B$. In particular, we conclude that $H$ is of infinite representation type ($Q_H$ is not a Dynkin quiver)
and $\mathcal{C}_T$ does not contain a projective module, and hence $T$ is without nonzero preinjective direct summands (see \cite[Proposition VIII.4.1]{ASS}).
Finally, we prove that $Q_H=\Delta^{\op}$ is a Euclidean quiver. Suppose that $Q_H$ is a wild quiver. Since $T$ has no nonzero preinjective direct summands,
it follows from \cite{Ke2} that $\Gamma_B$ admits an acyclic component $\Sigma$ with infinitely many $\tau_B$-orbits, with the stable part $\mathbb{Z}\mathbb{A}_{\infty}$,
contained entirely in the torsion-free part $\mathcal{Y}(T)$ of $\mo B$. Since $B$ is a cycle-finite algebra, $\Sigma$ consists of directing $B$-modules,
and hence $\Gamma_B$ contains infinitely many $\tau_B$-orbits containing directing modules, a contradiction.
Therefore, $Q_H$ is a Euclidean quiver and $B$ is a tilted algebra of Euclidean type $Q_H=\Delta^{\op}$. This finishes proof of the statements (i) and (ii).

In order to prove (iii), consider a module $N$ in $\mathcal D$ and an almost split sequence
\[ 0\to \tau_AN\to E\to N\to 0\]
in $\mo A$ with the right term $N$. Since $\mathcal D$ is left stable and closed under predecessors in $\mathcal C$, we have in $\mo A$ almost split sequences
\[ 0\to \tau^{m+1}_AN\to \tau^{m}_AE\to \tau^{m}_AN\to 0\]
for all nonnegative integers $m$. In particular, there exists a positive integer $n$ such that
\[ 0\to \tau^{n+1}_AN\to \tau^{n}_AE\to \tau^{n}_AN\to 0\]
is an exact sequence in the additive category $\add(\mathcal{Y}(T)\cap\mathcal{C}_T)=\add(\Gamma)$. Since $\mathcal{Y}(T)\cap\mathcal{C}_T=\Hom_H(T,Q(H))$,
this exact sequence is the image via the functor $\Hom_H(T,-): \mo H \to \mo B$ of an almost split sequence
\[ 0\to \tau_HU\to V\to U\to 0\]
with all terms in the additive category $\add(Q(H))$ of $Q(H)$. Then, applying Proposition \ref{prop21}, we conclude that
$\alpha(N)=\alpha(\tau^{n}_AN)= \alpha(\tau^{n}_BN)=\alpha(U)\leq 4$.
\end{proof}
\section{Proof of Theorem}
We will use the following results proved by S. Liu in \cite{Li3} (Theorem 7, Proposition 8, Lemma 6 and its dual).
\begin{thm}\label{th31}
Let $A$ be an algebra, and let
\[ 0\to \tau_AX\to\bigoplus_{i=1}^r Y_i\to X\to 0\]
be an almost split sequence in $\mo A$ with $Y_1, \ldots, Y_r$ from $\ind A$. Assume that one of the following conditions holds.
\begin{enumerate}
\renewcommand{\labelenumi}{\rm(\roman{enumi})}
\item $\tau_AX$ has a projective predecessor and $X$ has an injective successor in $\Gamma_A$.
\item $X$ lies on an oriented cycle in $\Gamma_A$.
\end{enumerate}
Then $r\leq 4$, and $r=4$ implies that one of the modules $Y_i$ is projective-injective, whereas the others are neither projective nor injective.
\end{thm}
\begin{prop}\label{prop32}
Let $A$ be an algebra, and let
\[ 0\to \tau_AX\to\bigoplus_{i=1}^r Y_i\to X\to 0\]
be an almost split sequence in $\mo A$ with $r\geq 5$ and $Y_1, \ldots, Y_r$ from $\ind A$. Then the following statements hold.
\begin{enumerate}
\renewcommand{\labelenumi}{\rm(\roman{enumi})}
\item If there is a sectional path from $\tau_AX$ to an injective module in $\Gamma_A$, then $\tau_AX$ has no projective predecessor in $\Gamma_A$.
\item If there is a sectional path from a projective module in $\Gamma_A$ to $X$, then $X$ has no injective successor in $\Gamma_A$.
\end{enumerate}
\end{prop}
We are now in position to prove the main result of the paper.

Let $A$ be a cycle-finite algebra, and let
\[ 0\to \tau_AX\to\bigoplus_{i=1}^r Y_i\to X\to 0\]
be an almost split sequence in $\mo A$ with $Y_1, \ldots, Y_r$ from $\ind A$, and let $\mathcal C$ be the component of $\Gamma_A$ containing $X$.
Assume $r\geq 5$. We claim that then $r=5$, one of the modules $Y_i$ is projective-injective, whereas the others are neither projective nor injective.

Since $r\geq 5$, it follows from Theorem \ref{th31} that $\tau_AX$ has no projective predecessor nor $X$ has no injective successor in $\Gamma_A$.
Assume that $\tau_AX$ has no projective predecessor in $\Gamma_A$.

We claim that then one of the modules $Y_i$ is projective. Suppose it is not the case. Then for any nonnegative integer $m$ we have in $\mo A$
an almost split sequence
\[ 0\to \tau^{m+1}_AX\to \bigoplus_{i=1}^r\tau^{m}_AY_i\to \tau^{m}_AX\to 0\]
with $r\geq 5$ and $\tau^{m}_AY_1, \ldots, \tau^{m}_AY_r$ from $\ind A$, because $\tau_AX$ has no projective predecessor in $\Gamma_A$. Moreover,
it follows from Theorem \ref{th31}, that $\tau_A^mX$, $m\geq 0$, are acyclic modules in $\Gamma_A$. Then it follows from \cite[Theorem 3.4]{Li1}
that the modules $\tau_A^mX$, $m\geq 0$, belong to an acyclic left stable full translation subquiver $\mathcal D$ of $\mathcal C$ which is closed
under predecessors. But then the assumption $r\geq 5$ contradicts Theorem \ref{th22}(iii). Therefore, one of the modules $Y_i$, say $Y_r$ is projective.

Observe now that the remaining modules $Y_1, \ldots, Y_{r-1}$ are noninjective. Indeed, since $Y_r$ is projective, we have $\ell(\tau_AX)<\ell(Y_r)$
and consequently $\sum_{i=1}^{r-1}\ell(Y_i)<\ell(X)$. Further, $Y_r$ is a projective predecessor of $X$ in $\Gamma_A$, and hence, applying Proposition
\ref{prop32}(ii), we conclude that $X$ has no injective successors in $\Gamma_A$. We claim that $Y_r$ is injective. Indeed, if it is not the case, we
have in $\mo A$ almost split sequences
\[ 0\to \tau^{-m+1}_AX\to \bigoplus_{i=1}^r\tau^{-m}_AY_i\to \tau^{-m}_AX\to 0\]
for all nonnegative integers $m$. Then, applying the dual of Theorem \ref{th22}, we obtain a contradiction with $r\geq 5$. Thus $Y_r$ is projective-injective.
Observe that then the modules $Y_1, \ldots, Y_{r-1}$ are nonprojective, because $Y_r$ injective forces the inequalities $\ell(X)<\ell(Y_r)$ and
$\sum_{i=1}^{r-1}\ell(Y_i)<\ell(\tau_AX)$.

Finally, since $\tau_AX$ has no projective predecessor in $\Gamma_A$, we have in $\mo A$ almost split sequences
\[ 0\to \tau^{m+1}_AX\to \bigoplus_{i=1}^{r-1}\tau^{m}_AY_i\to \tau^{m}_AX\to 0\]
for all positive integers $m$. Applying Proposition \ref{prop32} again, we conclude (as in the first part of the proof) that $r-1\leq 4$, and hence $r\leq 5$.
Therefore, $\alpha(X)=r=5$, one of the modules $Y_i$ is projective-injective, whereas the others are neither projective nor injective.
Moreover, if $Y_i$ is a projective-injective module, then $X\cong Y_i/\soc(Y_i)$.

\end{document}